\documentclass[12pt]{amsart}
\usepackage[margin=2cm]{geometry}
\usepackage{latexsym}
\usepackage{amsmath}
\usepackage{amsfonts}
\usepackage{amssymb}
\usepackage[all]{xy}
\usepackage{pgf,tikz}
\usetikzlibrary{arrows}

\newtheorem{theorem}{Theorem}[section]
\newtheorem{lemma}[theorem]{Lemma}

\newtheorem{conj}[theorem]{Conjecture}

\theoremstyle{definition}
\newtheorem{definition}[theorem]{Definition}

\theoremstyle{remark}

\numberwithin{equation}{section}

\begin{document}
\definecolor{zzttqq}{rgb}{0.6,0.2,0}
\definecolor{qqqqff}{rgb}{0,0,1}
\definecolor{darkspringgreen}{rgb}{0.09, 0.45, 0.27}
\title[MGS for the n-torus]{Maximal Green Sequences for Cluster Algebras Associated to the $n$-Torus with arbitrary punctures}

\author[Bucher]{Eric Bucher}
\address{Mathematics Department\\
Louisiana State University\\
Baton Rouge, Louisiana}
\email{ebuche2@tigers.lsu.edu}
\author[Mills]{Matthew R. Mills}
\address{Department of Mathematics, Wayne State University, Detroit, MI 48202}
\email{matthew.mills2@wayne.edu}

\thanks{{MM is supported by the GAANN Fellowship.}}
\thanks{{EB is supported by the GAANN Fellowship.}}
\subjclass{13F60, 30F99, 32G15}
\date{March, 2015}

\begin{abstract}
%this is your abstract
It is well known that any triangulation of a  marked surface produces a quiver. In this paper we will provide a triangulation for orientable surfaces of genus $n$ with an arbitrary number interior marked points (called punctures) whose corresponding quiver has a maximal green sequence. 

\end{abstract}

\maketitle

\section{Introduction}
\label{intro}

Cluster algebras were introduced by Fomin and Zelevinsky in \cite{cluster1}. Within a very short period of time cluster algebras became an important tool in the study of phenomena in various areas of mathematics and mathematical physics. They play an important role in the study of Teichm\"{u}ller theory, canonical bases, total positivity, Poisson Lie-groups, Calabi-Yau algebras, noncommutative Donaldson-Thomas invariants, scattering amplitudes,  and representations of finite dimensional algebras. For more information on the diverse scope of cluster algebras see the review paper by Williams \cite{williams}. 

A question that has arose that carries with it some ramifications is whether a quiver associated to a cluster algebra has a maximal green sequence. The idea of maximal green sequences of cluster mutations was introduced by Keller in \cite{keller}. He explored important applications of this notion, by utilizing it in the explicit computation of noncommutative Donaldson-Thomas invariants of triangulated categories which were introduced by Kontsevich and Soibelman in \cite{kontsevich}. If a quiver with potential has a maximal green sequence, then its associated Jacobi algebra is finite dimensional, this fact follows from Theorem 5.4 \cite{keller2} and Theorem 8.1 \cite{brustle}. In the work of Amoit \cite{amoit} there is given an explicit construction of a cluster category given a quiver with potential whose Jacobian algebra is finite dimensional. This means that finding a maximal green sequence for a given quiver allows us to categorify the associated cluster algebra. Muller showed that in fact the existence of a maximal green sequence corresponding to a quiver is not a mutation invariant \cite{muller}, which makes it all the more difficult to show that a maximal green sequence exists. You must make a strategic choice of initial seed (or initial quiver) and then find the sequence. The iterative nature mutations means that exhaustive methods are not always effective when searching for a maximal green sequence. There has been some progress made, and in the final section of this paper we will discuss which cluster algebras are known to have or not have associated maximal green sequences.
\\
\begin{table}
\begin{tabular}{|c|c|c|}
\hline
Surface Classification & Exists a MGS &Reference \\ 
\hline
Surface with boundary & Yes & Alim et al., \cite{alim} \\
Closed surface with $p=1$ & No & S. Ladkani, \cite{ladkani}\\
Sphere with $p\geq 4$ & Yes &  Alim et al., \cite{alim}\\
Closed surface with $g=1$ and $p \geq 2$ & Yes & Alim et al. \cite{alim}\\
Closed surface with $g\geq 2$ and  $p = 2$ & Yes & Bucher, \cite{bucher}\\
Closed surface with $g\geq 2$ and  $p\geq 3$ & Yes & Bucher and Mills\\
\hline
\end{tabular}
\label{references}
\vspace{3pt}
\caption{The different cases of surfaces, and whether or not they have a maximal green sequence. The letter $g$ denotes the genus of the surface, and $p$ denotes the number of punctures. By "closed surface" we mean a surface with no boundary component.}
\end{table}
The main result in this paper focuses on cluster algebras that are associated to triangulations of surfaces. This association is introduced by Gekhtman, Shapiro, and Vainshtein in \cite{gekhtman} and in a more general setting by Fock and Goncharov in \cite{fock}. This construction is extremely important because all but a few exceptional cases of cluster algebras of finite mutation type can be realized as a cluster algebra which arises from a surface following this construction. For a more in depth look into the procedure of creating a cluster algebra from a triangulated surface see the work by Fomin, Shapiro, and Thurston \cite{fomin}. 

Previous work by the first author in \cite{bucher} constructs maximal green sequences for cluster algebras which arise from surfaces with no empty boundary component and two punctures. In this paper we will push this result further, and construct maximal green sequences for cluster algebras arising from orientable surfaces with arbitrary genus and three or more punctures. We have provided references for all known cases in Table \ref{references}. This brings us to the statement of the main theorem for our paper. 
\begin{theorem}\label{main_theorem}
For a surface with no boundary, genus at least two, and with three or more punctures there exists a triangulation whose associated quiver has a maximal green sequence. 
\end{theorem}

The proof of Theorem \ref{main_theorem} is given in section \ref{proof_section}, but we will give a sketch of the proof here. We start with an $n$-torus and then construct a specific triangulation of this surface for which we can find a maximal green sequence for the associated cluster algebra. This triangulation is a modification of the triangulation that was used in \cite{bucher}, and will be given in section 3. After constructing the triangulation, we look at its associated quiver. We take advantage of the symmetry of this quiver by breaking it into smaller parts, finding maximal green sequences for these parts, and then carefully putting these sequences together. By adding interior punctures to the structure we add a ladder structure to our quiver. Increasing the number of punctures lengthens the ladder but its connection to the rest of the quiver is unaffected. We then induct on the number of punctures to finish the proof. 

As an immediate corollary to our result we also have the following result. 
\begin{theorem}\label{main2}
For every surface, it is known whether or not there exists a triangulation whose associated quiver has a maximal green sequence. 
\end{theorem}
\begin{proof}
Surfaces are determined by the genus of the surface, the number of boundary components, the number of punctures, and the number of marked points on each boundary component. Checking the surfaces given in Table \ref{references}, we see that all possible cases are accounted for. 
\end{proof}

In section 2 we give background on quivers and maximal green sequences. In section 5 we briefly discuss future work to be done with maximal green sequences. 

\section{Preliminaries}

We will follow the notation laid out by Br\"{u}stle, Dupont, and Perotin \cite{brustle}.
\begin{definition} A \textbf{quiver}, $Q$, is a directed graph containing no $2$-cycles or loops.
\end{definition}

The notation $Q_0$ will denote the vertices of $Q$. Also, $Q_1$ will denote the edges of $Q$ which are referred to as \textbf{arrows}. We will let $Q_0=[N]$.

\begin{definition} An \textbf{ice quiver} is a pair $(Q,F)$ where Q is a quiver as described above and $F \subset Q_0$ is a subset of vertices called frozen vertices; such that there are no arrows between them. For simplicity, we always assume that $Q_0=\{1,2,3, \dots ,n+m\}$ and that $F=\{n+1,n+2,\dots, n+m\}$ for some integers $n,m \geq 0$. If $F$ is empty we write $(Q,\emptyset)$ for the ice quiver. 
\end{definition}
In this paper we will be concerned with a process called mutation. Mutation is a process of obtaining a new ice quiver from an existing one. 
\begin{definition} Let $(Q,F)$ be an ice quiver and $k\in Q_0$ a non-frozen vertex. The \textbf{mutation} of a quiver $(Q,F)$ at a vertex $k$ is denoted $\mu_k$, and produces a new ice quiver $(\mu_k(Q),F)$. The vertices of $(\mu_k(Q),F)$ are the same vertices from $(Q,F)$. The arrows of the new quiver are obtained by performing the following $3$ steps:
\begin{enumerate}
\item For every 2-path $i \rightarrow k \rightarrow j$ , adjoin a new arrow $i \rightarrow j$.
\item Reverse the direction of all arrows incident to $k$.
\item Delete any 2-cycles created during the first two steps as well as any arrows created between frozen vertices.
\end{enumerate}
\end{definition}

It is important to note that we do not allow mutation at a frozen vertex. We will denote $Mut(Q)$ to be the set of all quivers who can be obtained from $Q$ by a sequence of mutations.

The ice quivers which are of concern in this paper have a very specific set of frozen vertices. We will be looking at what are referred to as the framed and coframed quivers associated to $Q$.

\begin{definition} 
The \textbf{framed quiver} associated with $Q$ is the quiver $\hat{Q}$ such that: 

$$\hat{Q}_0 = Q_0 \sqcup \{i'\text{ }|\text{ }i\in Q_0\}$$
$$\hat{Q}_1 = Q_1 \sqcup \{i \to i'\text{ }|\text{ }i \in Q_0\}$$

\flushleft The \textbf{coframed quiver} associated with $Q$ is the quiver $\breve{Q}$ such that:
$$\breve{Q}_0 = Q_0 \sqcup \{i'\text{ }|\text{ }i\in Q_0\}$$
$$\breve{Q}_1 = Q_1 \sqcup \{i' \to i\text{ }|\text{ }i \in Q_0\}$$
\end{definition}
Both quivers $\hat{Q}$ and $\breve{Q}$ are naturally ice quivers whose frozen vertices are commonly written as $\hat{Q}_0'$ and $\breve{Q}_0'$. Next we will talk about what it means for a vertex to be green or red.

\begin{definition}
Let $R \in Mut(\hat{Q})$. A non-frozen vertex $i \in R_0$ is called \textbf{green} if $$\{j'\in Q_0'\text{ }| \text{ } \exists \text{ } j' \rightarrow i \in R_1 \}=\emptyset.$$ It is called \textbf{red} if $$\{j'\in Q_0'\text{ }| \text{ } \exists \text{ } j' \leftarrow i \in R_1 \}=\emptyset.$$ 
\end{definition}

In  they show that every non-frozen vertex in $R_0$ is either red or green. This idea is what motivates our work in this paper. It arises as a question of green sequences. 

\begin{definition}
A \textbf{green sequence} for $Q$ is a sequence $\textbf{i}=\{i_1, \dots, i_l\} \subset Q_0$ such that $i_1$ is green in $\hat{Q}$ and for any $2\leq k \leq l$, the vertex $i_k$ is green in $\mu_{i_{k-1}}\circ \cdots \circ \mu_{i_1}(\hat{Q})$. The integer $l$ is called the length of the sequence $\textbf{i}$ and is denoted by $l(\textbf{i})$.

A green sequence \textbf{i} is called maximal if every non-frozen vertex in $\mu_{\textbf{i}}(\hat{Q})$ is red where $\mu_{\textbf{i}}=\mu_{i_{l}}\circ \cdots \circ \mu_{i_1}$. We denote the set of all maximal green sequences for $Q$ by $$\text{green}(Q)=\{\textbf{i } | \textbf{ i}\text{ is a maximal green sequence for }Q\}.$$  
\end{definition}

In this paper we will construct a maximal green sequence for a specific infinite family of quivers which will be described in the following section. In essence what we want to show is that green$(Q)\neq \emptyset$ for each quiver, $Q$, in this family. One important observation that the proof of our main result relies on is that if $Q_1$ and $Q_2$ are quivers such that $Q_1$ has a maximal green sequence $\lambda$, and $Q_1$ is a full subquiver of $Q_2$ consisting of only green vertices, then $\lambda$ is a green sequence for $Q_2$. 

\section{Constructing the Triangulation}
Following the work done by Fomin, Shapiro, and Thurston in \cite{fomin} we construct a quiver $Q_n^p$ associated to the genus $n$ surface with no boundary and $p\geq 3$ punctures. The triangulation of the surface with two punctures that was given in \cite{bucher} (see Figure \ref{triangulation2punc}) gives rise to a very natural generalization. When $p\geq 3$ we can remove the arc $f_n$ and replace it with a $(p-2)$-times punctured digon. We then triangulate the digon as shown on the right of Figure \ref{triangulation}. The final triangulation is also given in Figure \ref{triangulation}. 

From this triangulation we can construct the quiver $Q_n^p$. The quiver associated to the 3-torus with 7 punctures is given on the left of Figure \ref{quiver}. Also we define the quiver $P^{p-3}$ for $p\geq 3$ to be the full subquiver of $Q_n^p$ consisting of the vertices $\{g_0,g_1^1,g_2^1,g_3^1,\ldots,g_1^{p-3},g_2^{p-3},g_3^{p-3}\}$. $P^4$ is given on the right of Figure \ref{quiver}.
Note that by increasing the genus of the surface the cycle containing the $f$ vertices gets longer, and more handles are added. Increasing the number of punctures will increase the number of rows in the $P$ subquiver. The important thing to notice is that the fundamental shape of the quiver $Q_n^p$ doesn't change. 
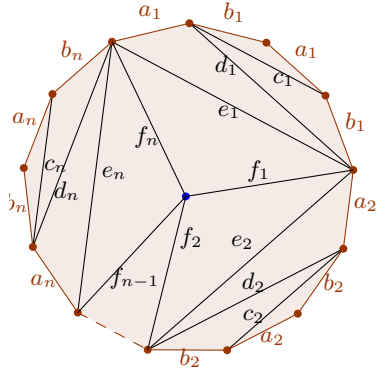
\begin{figure}
\[\begin{tikzpicture}[line cap=round,line join=round,>=triangle 45,x=0.3cm,y=0.3cm]
\clip(-11.54,-21.4) rectangle (5.26,-4.41);
\fill[color=zzttqq,fill=zzttqq,fill opacity=0.1] (-5.42,-20.19) -- (-1.9,-20.21) -- (1.23,-18.59) -- (3.25,-15.71) -- (3.69,-12.22) -- (2.46,-8.92) -- (-0.16,-6.57) -- (-3.57,-5.71) -- (-7,-6.53) -- (-9.64,-8.85) -- (-10.91,-12.13) -- (-10.51,-15.63) -- (-8.52,-18.54) -- cycle;
\draw [color=zzttqq] (-5.42,-20.19)-- (-1.9,-20.21);
\draw [color=zzttqq] (-1.9,-20.21)-- (1.23,-18.59);
\draw [color=zzttqq] (1.23,-18.59)-- (3.25,-15.71);
\draw [color=zzttqq] (3.25,-15.71)-- (3.69,-12.22);
\draw [color=zzttqq] (3.69,-12.22)-- (2.46,-8.92);
\draw [color=zzttqq] (2.46,-8.92)-- (-0.16,-6.57);
\draw [color=zzttqq] (-0.16,-6.57)-- (-3.57,-5.71);
\draw [color=zzttqq] (-3.57,-5.71)-- (-7,-6.53);
\draw [color=zzttqq] (-7,-6.53)-- (-9.64,-8.85);
\draw [color=zzttqq] (-9.64,-8.85)-- (-10.91,-12.13);
\draw [color=zzttqq] (-10.91,-12.13)-- (-10.51,-15.63);
\draw [color=zzttqq] (-10.51,-15.63)-- (-8.52,-18.54);
\draw [dash pattern=on 4pt off 4pt,color=zzttqq] (-8.52,-18.54)-- (-5.42,-20.19);
\draw (-7,-6.53)-- (3.69,-12.22);
\draw (3.69,-12.22)-- (-3.57,-5.71);
\draw (-3.57,-5.71)-- (2.46,-8.92);
\draw (3.69,-12.22)-- (-5.42,-20.19);
\draw (-5.42,-20.19)-- (3.25,-15.71);
\draw (3.25,-15.71)-- (-1.9,-20.21);
\draw (-8.52,-18.54)-- (-7,-6.53);
\draw (-7,-6.53)-- (-10.51,-15.63);
\draw (-10.51,-15.63)-- (-9.64,-8.85);
\begin{scriptsize}
\fill [color=zzttqq] (-5.42,-20.19) circle (1.5pt);
\fill [color=zzttqq] (-1.9,-20.21) circle (1.5pt);
\draw[color=zzttqq] (-3.53,-20.54) node {$b_2$};
\draw[color=zzttqq] (0.1,-19.68) node {$a_2$};
\draw[color=zzttqq] (2.84,-17.22) node {$b_2$};
\draw[color=zzttqq] (4.24,-13.74) node {$a_2$};
\draw[color=zzttqq] (3.81,-10.07) node {$b_1$};
\draw[color=zzttqq] (1.63,-6.99) node {$a_1$};
\draw[color=zzttqq] (-1.58,-5.23) node {$b_1$};
\draw[color=zzttqq] (-5.29,-5.23) node {$a_1$};
\draw[color=zzttqq] (-8.73,-6.91) node {$b_n$};
\draw[color=zzttqq] (-10.84,-9.96) node {$a_n$};
\draw[color=zzttqq] (-11.23,-13.67) node {$b_n$};
\draw[color=zzttqq] (-10.02,-17.14) node {$a_n$};
\draw[color=zzttqq] (-7.09,-19.64) node {};
\fill [color=zzttqq] (1.23,-18.59) circle (1.5pt);
\fill [color=zzttqq] (3.25,-15.71) circle (1.5pt);
\fill [color=zzttqq] (3.69,-12.22) circle (1.5pt);
\fill [color=zzttqq] (2.46,-8.92) circle (1.5pt);
\fill [color=zzttqq] (-0.16,-6.57) circle (1.5pt);
\fill [color=zzttqq] (-3.57,-5.71) circle (1.5pt);
\fill [color=zzttqq] (-7,-6.53) circle (1.5pt);
\fill [color=zzttqq] (-9.64,-8.85) circle (1.5pt);
\fill [color=zzttqq] (-10.91,-12.13) circle (1.5pt);
\fill [color=zzttqq] (-10.51,-15.63) circle (1.5pt);
\fill [color=zzttqq] (-8.52,-18.54) circle (1.5pt);
\fill [color=qqqqff] (-3.73,-13.39) circle (1.5pt);
\draw[color=black] (-1.81,-9.64) node {$e_1$};
\draw[color=black] (0.61,-8.2) node {$c_1$};
\draw[color=black] (-1.92,-7.57) node {$d_1$};
\draw[color=black] (-1.23,-15.42) node {$e_2$};
\draw[color=black] (-0.72,-17.2) node {$d_2$};
\draw[color=black] (-0.72,-18.7) node {$c_2$};
\draw[color=black] (-6.9,-12.34) node {$e_n$};
\draw[color=black] (-9,-13.15) node {$d_n$};
\draw[color=black] (-9.48,-12.03) node {$c_n$};

\draw[color=black] (-3.73,-13.39)-- (3.69,-12.22);
\draw[color=black] (-3.73,-13.39)-- (-5.42,-20.19);
\draw[color=black] (-3.73,-13.39)-- (-8.52,-18.54);
\draw[color=black] (-3.73,-13.39)-- (-7,-6.53);
\draw[color=black] (-.5,-12.25) node {$f_1$};
\draw[color=black] (-3.5,-15.25) node {$f_2$};
\draw[color=black] (-6,-17) node {$f_{n-1}$};
\draw[color=black] (-5.45,-10.75) node {$f_n$};
\end{scriptsize}
\end{tikzpicture}
\]
\caption{A triangulation of an $n$-torus with two punctures.}
\label{triangulation2punc}
\end{figure}

\begin{figure}
\begin{center}
$\begin{tikzpicture}[line cap=round,line join=round,>=triangle 45,x=0.3cm,y=0.3cm]
\clip(-11.54,-21.4) rectangle (5.26,-4.41);
\fill[color=zzttqq,fill=zzttqq,fill opacity=0.1] (-5.42,-20.19) -- (-1.9,-20.21) -- (1.23,-18.59) -- (3.25,-15.71) -- (3.69,-12.22) -- (2.46,-8.92) -- (-0.16,-6.57) -- (-3.57,-5.71) -- (-7,-6.53) -- (-9.64,-8.85) -- (-10.91,-12.13) -- (-10.51,-15.63) -- (-8.52,-18.54) -- cycle;
\draw [color=zzttqq] (-5.42,-20.19)-- (-1.9,-20.21);
\draw [color=zzttqq] (-1.9,-20.21)-- (1.23,-18.59);
\draw [color=zzttqq] (1.23,-18.59)-- (3.25,-15.71);
\draw [color=zzttqq] (3.25,-15.71)-- (3.69,-12.22);
\draw [color=zzttqq] (3.69,-12.22)-- (2.46,-8.92);
\draw [color=zzttqq] (2.46,-8.92)-- (-0.16,-6.57);
\draw [color=zzttqq] (-0.16,-6.57)-- (-3.57,-5.71);
\draw [color=zzttqq] (-3.57,-5.71)-- (-7,-6.53);
\draw [color=zzttqq] (-7,-6.53)-- (-9.64,-8.85);
\draw [color=zzttqq] (-9.64,-8.85)-- (-10.91,-12.13);
\draw [color=zzttqq] (-10.91,-12.13)-- (-10.51,-15.63);
\draw [color=zzttqq] (-10.51,-15.63)-- (-8.52,-18.54);
\draw [dash pattern=on 4pt off 4pt,color=zzttqq] (-8.52,-18.54)-- (-5.42,-20.19);
\draw (-7,-6.53)-- (3.69,-12.22);
\draw (3.69,-12.22)-- (-3.57,-5.71);
\draw (-3.57,-5.71)-- (2.46,-8.92);
\draw (3.69,-12.22)-- (-5.42,-20.19);
\draw (-5.42,-20.19)-- (3.25,-15.71);
\draw (3.25,-15.71)-- (-1.9,-20.21);
\draw (-8.52,-18.54)-- (-7,-6.53);
\draw (-7,-6.53)-- (-10.51,-15.63);
\draw (-10.51,-15.63)-- (-9.64,-8.85);
\begin{scriptsize}
\fill [color=zzttqq] (-5.42,-20.19) circle (1.5pt);
\fill [color=zzttqq] (-1.9,-20.21) circle (1.5pt);
\draw[color=zzttqq] (-3.53,-20.70) node {$b_2$};
\draw[color=zzttqq] (0.1,-19.68) node {$a_2$};
\draw[color=zzttqq] (2.84,-17.22) node {$b_2$};
\draw[color=zzttqq] (4.24,-13.74) node {$a_2$};
\draw[color=zzttqq] (3.81,-10.07) node {$b_1$};
\draw[color=zzttqq] (1.63,-6.99) node {$a_1$};
\draw[color=zzttqq] (-1.58,-5.23) node {$b_1$};
\draw[color=zzttqq] (-5.29,-5.23) node {$a_1$};
\draw[color=zzttqq] (-8.73,-6.91) node {$b_n$};
\draw[color=zzttqq] (-10.84,-9.96) node {$a_n$};
\draw[color=zzttqq] (-11.23,-13.67) node {$b_n$};
\draw[color=zzttqq] (-10.02,-17.14) node {$a_n$};
\draw[color=zzttqq] (-7.09,-19.64) node {};
\fill [color=zzttqq] (1.23,-18.59) circle (1.5pt);
\fill [color=zzttqq] (3.25,-15.71) circle (1.5pt);
\fill [color=zzttqq] (3.69,-12.22) circle (1.5pt);
\fill [color=zzttqq] (2.46,-8.92) circle (1.5pt);
\fill [color=zzttqq] (-0.16,-6.57) circle (1.5pt);
\fill [color=zzttqq] (-3.57,-5.71) circle (1.5pt);
\fill [color=zzttqq] (-7,-6.53) circle (1.5pt);
\fill [color=zzttqq] (-9.64,-8.85) circle (1.5pt);
\fill [color=zzttqq] (-10.91,-12.13) circle (1.5pt);
\fill [color=zzttqq] (-10.51,-15.63) circle (1.5pt);
\fill [color=zzttqq] (-8.52,-18.54) circle (1.5pt);
\fill [color=qqqqff] (-3.73,-13.39) circle (1.5pt);
%\draw[color=black] (-1.81,-9.64) node {$e_1$};
%\draw[color=black] (0.61,-8.2) node {$c_1$};
%\draw[color=black] (-1.92,-7.57) node {$d_1$};
%\draw[color=black] (-1.23,-15.42) node {$e_2$};
%\draw[color=black] (-0.72,-17.2) node {$d_2$};
%\draw[color=black] (-0.72,-18.7) node {$c_2$};
%\draw[color=black] (-6.9,-12.34) node {$e_n$};
%\draw[color=black] (-9,-13.15) node {$d_n$};
%\draw[color=black] (-9.48,-12.03) node {$c_n$};

\draw[color=black] (-3.73,-13.39)-- (3.69,-12.22);
\draw[color=black] (-3.73,-13.39)-- (-5.42,-20.19);
\draw[color=black] (-3.73,-13.39)-- (-8.52,-18.54);
\draw[color=blue] (-3.73,-13.39)-- (-4.55,-11.67);
\draw[color=blue] (-5.37,-9.95)-- (-7,-6.53);
\draw[color=black] (-.5,-13.75) node {$f_1$};
%\draw[color=black] (-3.5,-15.25) node {$f_2$};
\draw[color=black] (-6,-17) node {$f_{n-1}$};

\fill [color=qqqqff] (-4.55,-11.67) circle (1.5pt);
\fill [color=qqqqff] (-5.37,-9.95) circle (1.5pt);
\fill [color=qqqqff] (-6.19,-8.23) circle (1.5pt);
\draw [dash pattern=on 4pt off 4pt,color=qqqqff] (-4.55,-11.67)-- (-5.37,-9.95);
\draw [color=qqqqff] (-3.73,-13.39) to[out=55,in=-35] (-7,-6.53);
\draw [color=qqqqff] (-7,-6.53) to[out=-90,in=-175] (-3.73,-13.39);

\draw [color=qqqqff] (-4.55,-11.67) to[out=55,in=-35] (-7,-6.53);
\draw [color=qqqqff] (-7,-6.53) to[out=-90,in=-175] (-4.55,-11.67);

\draw [color=qqqqff] (-5.37,-9.95) to[out=55,in=-35] (-7,-6.53);
\draw [color=qqqqff] (-7,-6.53) to[out=-90,in=-175] (-5.37,-9.95);

\draw[color=blue] (-5.45,-13.75) node {$f_n$};
%\draw[color=blue] (-4,-12.25) node {$f_{n+1}$};
\draw[color=blue] (-1.45,-11.75) node {$f_{n+2}$};

\end{scriptsize}
\end{tikzpicture}
$ \hspace{1cm}
$
\begin{tikzpicture}[x=0.7cm,y=0.7cm]
\coordinate (a) at (0,0);
\coordinate (b) at (10,0);
\coordinate (c) at (8,0);
\coordinate(d) at (2,0);
\coordinate (e) at (4,0);
\coordinate (f) at (6,0);
\fill[color=blue] (a) circle (1.5pt);
\fill[color=blue] (b) circle (1.5pt);
\fill[color=blue] (c) circle (1.5pt);
\fill[color=blue] (d) circle (1.5pt);
\fill[color=blue] (e) circle (1.5pt);
\fill[color=blue] (f) circle (1.5pt);
\draw[color=blue] (f) -- (b);
\draw[color=blue] (a) -- (e);
\draw[color=blue] [dash pattern=on 4pt off 4pt] (e) -- (f);
\draw[color=blue] (a) to[out=35,in=105] (b);
\draw[color=blue] (a) to[out=-35,in=-105] (b);
\draw[color=blue] (a) to[out=35,in=105] (c);
\draw[color=blue] (a) to[out=-35,in=-105] (c);
\draw[color=blue] (a) to[out=35,in=105] (e);
\draw[color=blue] (a) to[out=-35,in=-105] (e);
\draw[color=blue] (a) to[out=35,in=105] (f);
\draw[color=blue] (a) to[out=-35,in=-105] (f);
\begin{scriptsize}
\draw[color=blue] (1,.2) node {$g_0$};
\draw[color=blue] (3,.3) node {$g_1^1$};
\draw[color=blue] (4,-.8) node {$g_1^1$};
\draw[color=blue] (4,.8) node {$g_3^1$};
\draw[color=blue] (10,1.75) node {$f_{n+2}$};
\draw[color=blue] (10,-1.75) node {$f_{n}$};
\draw [color=blue](9,.2) node {$f_{n+1}$};
\draw[color=blue] (8,1.25) node {$g_3^{p-3}$};
\draw[color=blue] (8,-1.25) node {$g_1^{p-3}$};
\draw[color=blue] (7,.3) node {$g_2^{p-3}$};
\end{scriptsize}
\end{tikzpicture}$
\end{center}
\caption{A triangulation of $n$-torus with $p$ punctures.}\label{triangulation}
\end{figure}
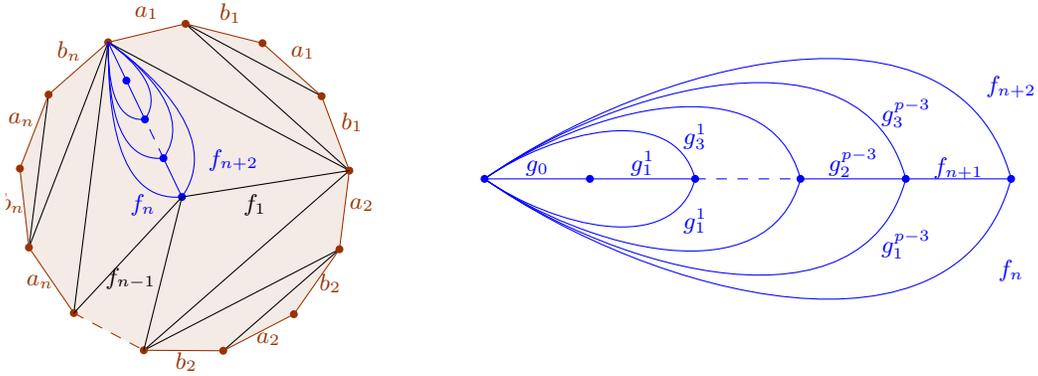

 \begin{figure} \begin{center}
 $\begin{xy} 0;<.5pt,0pt>:<0pt,-.5pt>:: 
(75,225) *+{f_3} ="0",
(125,175) *+{f_2} ="1",
(275,175) *+{f_1} ="2",
(325,225) *+{f_5} ="3",
(200,225) *+{f_4} ="4",
(125,375) *+{g_1^2} ="5",
(50,150) *+{e_3} ="6",
(200,125) *+{e_2} ="7",
(350,150) *+{e_1} ="8",
(275,375) *+{g_3^2} ="9",
(200,375) *+{g_2^2} ="10",
(200,425) *+{g_1^1} ="11",
(0,100) *+{d_3} ="12",
(100,100) *+{a_3} ="13",
(50,0) *+{b_3} ="14",
(50,75) *+{c_3} ="15",
(150,75) *+{d_2} ="16",
(250,75) *+{a_2} ="17",
(200,0) *+{b_2} ="18",
(200,50) *+{c_2} ="19",
(300,100) *+{d_1} ="20",
(400,100) *+{a_1} ="21",
(350,0) *+{b_1} ="22",
(350,75) *+{c_1} ="23",
(125,425) *+{g_1^1} ="24",
(200,475) *+{g_0} ="25",
(275,425) *+{g_3^1} ="26",
(125,325) *+{g_1^3} ="27",
(200,325) *+{g_2^3} ="28",
(275,325) *+{g_3^3} ="29",
(125,275) *+{g_1^4} ="30",
(200,275) *+{g_2^4} ="31",
(275,275) *+{g_3^4} ="32",
"0", {\ar"1"},
"4", {\ar"0"},
"6", {\ar"0"},
"0", {\ar"30"},
"1", {\ar"2"},
"1", {\ar"6"},
"7", {\ar"1"},
"2", {\ar"3"},
"2", {\ar"7"},
"8", {\ar"2"},
"3", {\ar"4"},
"3", {\ar"8"},
"32", {\ar"3"},
"30", {\ar"4"},
"4", {\ar"32"},
"10", {\ar"5"},
"5", {\ar"24"},
"27", {\ar"5"},
"5", {\ar"28"},
"12", {\ar"6"},
"6", {\ar"13"},
"16", {\ar"7"},
"7", {\ar"17"},
"20", {\ar"8"},
"8", {\ar"21"},
"9", {\ar"10"},
"26", {\ar"9"},
"28", {\ar"9"},
"9", {\ar"29"},
"24", {\ar"10"},
"10", {\ar"26"},
"24", {\ar"11"},
"11", {\ar"26"},
"13", {\ar"12"},
"14", {\ar"12"},
"12", {\ar"15"},
"14", {\ar"13"},
"13", {\ar"15"},
"15", {\ar|*+{\scriptstyle 2}"14"},
"17", {\ar"16"},
"18", {\ar"16"},
"16", {\ar"19"},
"18", {\ar"17"},
"17", {\ar"19"},
"19", {\ar|*+{\scriptstyle 2}"18"},
"21", {\ar"20"},
"22", {\ar"20"},
"20", {\ar"23"},
"22", {\ar"21"},
"21", {\ar"23"},
"23", {\ar|*+{\scriptstyle 2}"22"},
"25", {\ar"24"},
"26", {\ar"25"},
"28", {\ar"27"},
"30", {\ar"27"},
"27", {\ar"31"},
"29", {\ar"28"},
"31", {\ar"29"},
"29", {\ar"32"},
"31", {\ar"30"},
"32", {\ar"31"},
\end{xy}$
$\begin{xy} 0;<.5pt,0pt>:<0pt,-.5pt>:: 
(75,300) *+{g_0} ="0",
(0,225) *+{g_1^1} ="1",
(75,225) *+{g_2^1} ="2",
(150,225) *+{g_3^1} ="3",
(0,150) *+{g_1^2} ="4",
(75,150) *+{g_2^2} ="5",
(150,150) *+{g_3^2} ="6",
(0,75) *+{g_1^3} ="7",
(75,75) *+{g_2^3} ="8",
(150,75) *+{g_3^3} ="9",
(0,0) *+{g_1^4} ="10",
(75,0) *+{g_2^4} ="11",
(150,0) *+{g_3^4} ="12",
"1", {\ar"0"},
"0", {\ar"3"},
"2", {\ar"1"},
"4", {\ar"1"},
"1", {\ar"5"},
"3", {\ar"2"},
"5", {\ar"3"},
"3", {\ar"6"},
"5", {\ar"4"},
"7", {\ar"4"},
"4", {\ar"8"},
"6", {\ar"5"},
"8", {\ar"6"},
"6", {\ar"9"},
"8", {\ar"7"},
"10", {\ar"7"},
"7", {\ar"11"},
"9", {\ar"8"},
"11", {\ar"9"},
"9", {\ar"12"},
"11", {\ar"10"},
"12", {\ar"11"},
\end{xy}$
\end{center}
\caption{The quivers $Q_3^7$ (left) and $P_4$ (right).}
\label{quiver}
\end{figure}
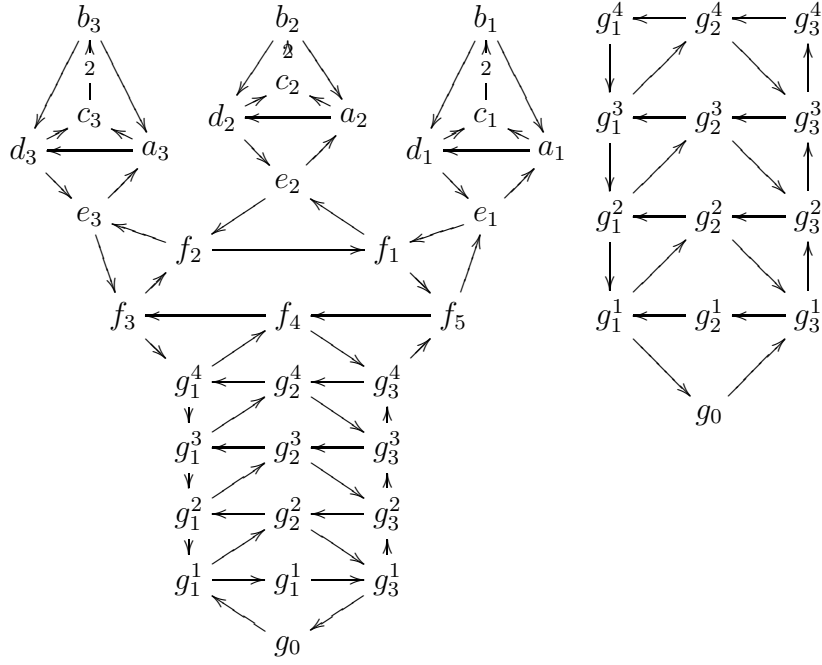

\section{Proof of Main Result}\label{proof_section}
We begin by recalling a result from \cite{bucher} that gives a maximal green sequence for oriented cycles. 
\begin{lemma}{\cite[Lemma 4.2]{bucher}} \label{cycle}
Let $C$ be a quiver that is an oriented $n$-cycle with vertices labeled $c_i$ $i=1,\ldots,n$, with $c_{i} \rightarrow c_{i-1}$ for $ 2 \leq i \leq n$ and $c_1 \rightarrow c_n$. Define a sequence $$\gamma(c_n,c_{n-1},\ldots,c_{2},c_1)=c_nc_{n-1}\cdots c_2c_1 c_3 c_4 \cdots c_{n-1} c_n.$$ Then $\gamma$ is a maximal green sequence for $C$. Furthermore, after applying $\gamma$ to $C$ the resultant quiver is still an oriented cycle with $c_2,$ and $c_1$ interchanged. 
\end{lemma}

We will rephrase the statement of Theorem \ref{main_theorem} to use the terminology we defined in the previous section. 
\begin{theorem}\label{main_v2}
Let $Q_n^p$ be the quiver obtained from our triangulation of a genus $n$ surface with no boundary and $p\geq 3$ punctures. Then $Q_n^p$ has a maximal green sequence given by 
$$\gamma(f_{n+2}f_{n+1}\cdots f_1)\sigma_n\cdots\sigma_1\alpha_0\alpha_1\cdots\alpha_{p-3} \gamma(f_{n+2},f_2,f_1,f_{3},f_4,\ldots,f_{n})f_{n+1} \beta_{p-3} \tau_1\tau_2\cdots\tau_n,$$
where $\gamma$ is defined as in Lemma \ref{cycle} and $\sigma, \tau, \alpha,$ and $\beta$ are defined as follows:
$$\sigma_i = e_id_ib_ic_ia_ib_id_ie_ic_ia_ib_i, $$ $$ \tau_i=e_ib_ia_ic_ie_id_ib_ia_ie_i,$$
$$\alpha_j=\begin{cases} 
g_0 & j=0 \\ 
g_3^1g_2^1g_1^1g_3^1g_0 & j=1 \\
g_3^2g_2^2g_2^1g_1^1g_1^2g_3^2g_2^1g_3^1g_0 & j=2 \\
g_3^j g_2^j g_1^{j-2} g_1^{j-1} g_1^{j} g_3^j g_1^{j-2} g_3^{j-1} g_1^{j-3} g_3^{j-2}\cdots g_1^1g_3^2g_2^1g_3^1g_0 & j\geq 3 \\ 
 \end{cases}$$
$$\beta_j=\begin{cases} 
\emptyset & j=0 \\ 
g_1^1g_2^1g_3^1 & j=1 \\
g_1^1g_1^2g_1^1g_3^2g_2^1g_3^1 & j=2 \\
g_1^{n-1}g_1^jg_1^{n-1}g_3^ng_1^{n-2}g_3^{n-1} g_1^{n-3}\cdots g_3^3 g_1^1g_3^2g_2^1g_3^1& j\geq 3 \\ 
 \end{cases}$$
\end{theorem}

\begin{lemma} \label{tail}
$P_n$ has a maximal green sequence given by $\alpha_0\alpha_1\cdots \alpha_n$.
\end{lemma} 
\begin{proof}
The sequence is easily checked for $n=0,1,2$. For n=3, apply $\alpha_0\alpha_1\alpha_2$ to $P_3$. We know that this is a green sequence for the $P_2$ subquiver of $P_3$. The current state of the quiver is given in the following diagram.
$$ 
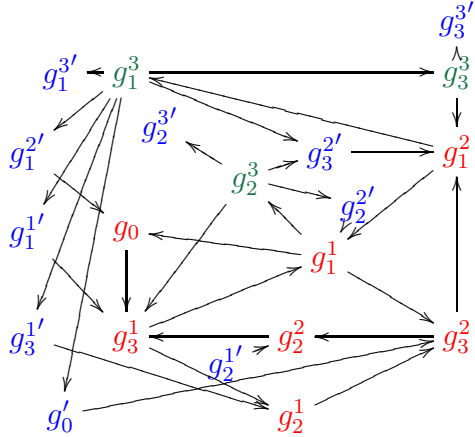
\begin{xy} 0;<.5pt,0pt>:<0pt,-.4pt>:: 
(75,50) *+{\color{darkspringgreen}{g_1^3}} ="0",
(165,150) *+{\color{darkspringgreen}{g_2^3}} ="1",
(325,50) *+{\color{darkspringgreen}{g_3^3}} ="2",
(325,125) *+{\color{red}{g_1^2}} ="3",
(200,300) *+{\color{red}{g_2^2} }="4",
(325,300) *+{\color{red}{g_3^2} }="5",
(225,225) *+{\color{red}{g_1^1} }="6",
(200,375) *+{\color{red}{g_2^1}} ="7",
(75,300) *+{\color{red}{g_3^1} }="8",
(75,200) *+{\color{red}{g_0}} ="9",
(25,50) *+{\color{blue}{{g_1^3}'}} ="10",
(100,100) *+{\color{blue}{{g_2^3}'} }="11",
(325,0) *+{\color{blue}{{g_3^3}'}} ="12",
(0,125) *+{\color{blue}{{g_1^2}'}} ="13",
(250,175) *+{\color{blue}{{g_2^2}'}} ="14",
(225,125) *+{\color{blue}{{g_3^2}'}} ="15",
(0,200) *+{\color{blue}{{g_1^1}'}} ="16",
(150,325) *+{\color{blue}{{g_2^1}'}} ="17",
(25,375) *+{\color{blue}{g_0'}} ="18",
(0,300) *+{\color{blue}{{g_3^1}'}} ="19",
"0", {\ar"2"},
"3", {\ar"0"},
"0", {\ar"10"},
"0", {\ar"13"},
"0", {\ar"15"},
"0", {\ar"16"},
"0", {\ar"18"},
"0", {\ar"19"},
"6", {\ar"1"},
"1", {\ar"8"},
"1", {\ar"11"},
"1", {\ar"14"},
"1", {\ar"15"},
"2", {\ar"3"},
"2", {\ar"12"},
"5", {\ar"3"},
"3", {\ar"6"},
"15", {\ar"3"},
"5", {\ar"4"},
"4", {\ar"8"},
"17", {\ar"4"},
"6", {\ar"5"},
"7", {\ar"5"},
"18", {\ar"5"},
"8", {\ar"6"},
"6", {\ar"9"},
"14", {\ar"6"},
"8", {\ar"7"},
"19", {\ar"7"},
"9", {\ar"8"},
"16", {\ar"8"},
"13", {\ar"9"},
\end{xy} $$
We now apply the first four mutations of $\alpha_3$ to the quiver above. 
$$\begin{array}{l r}
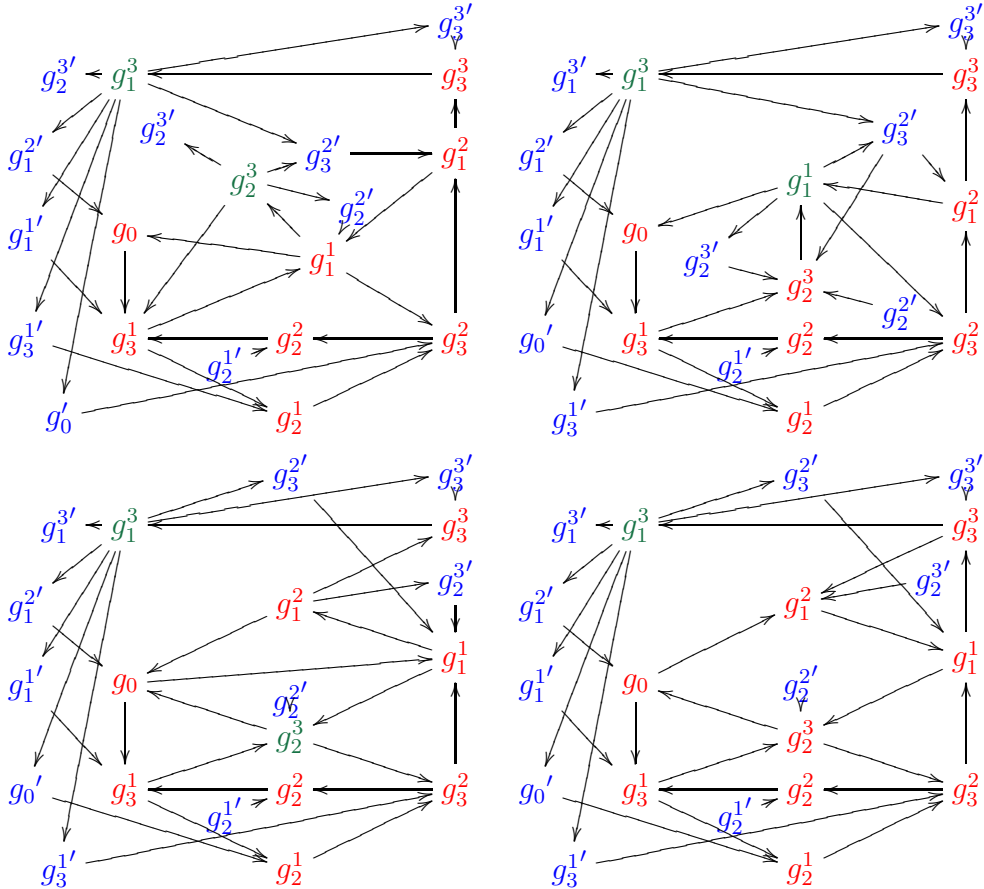
\begin{xy} 0;<.5pt,0pt>:<0pt,-.4pt>:: 
(75,50) *+{\color{darkspringgreen}{g_1^3}} ="0",
(165,150) *+{\color{darkspringgreen}{g_2^3}} ="1",
(325,50) *+{\color{red}{g_3^3}} ="2",
(325,125) *+{\color{red}{g_1^2}} ="3",
(200,300) *+{\color{red}{g_2^2} }="4",
(325,300) *+{\color{red}{g_3^2} }="5",
(225,225) *+{\color{red}{g_1^1} }="6",
(200,375) *+{\color{red}{g_2^1}} ="7",
(75,300) *+{\color{red}{g_3^1} }="8",
(75,200) *+{\color{red}{g_0}} ="9",
(25,50) *+{\color{blue}{{g_2^3}'}} ="10",
(100,100) *+{\color{blue}{{g_2^3}'} }="11",
(325,0) *+{\color{blue}{{g_3^3}'}} ="12",
(0,125) *+{\color{blue}{{g_1^2}'}} ="13",
(250,175) *+{\color{blue}{{g_2^2}'}} ="14",
(225,125) *+{\color{blue}{{g_3^2}'}} ="15",
(0,200) *+{\color{blue}{{g_1^1}'}} ="16",
(150,325) *+{\color{blue}{{g_2^1}'}} ="17",
(25,375) *+{\color{blue}{g_0'}} ="18",
(0,300) *+{\color{blue}{{g_3^1}'}} ="19",
"2", {\ar"0"},
"0", {\ar"10"},
"0", {\ar"12"},
"0", {\ar"13"},
"0", {\ar"15"},
"0", {\ar"16"},
"0", {\ar"18"},
"0", {\ar"19"},
"6", {\ar"1"},
"1", {\ar"8"},
"1", {\ar"11"},
"1", {\ar"14"},
"1", {\ar"15"},
"3", {\ar"2"},
"12", {\ar"2"},
"5", {\ar"3"},
"3", {\ar"6"},
"15", {\ar"3"},
"5", {\ar"4"},
"4", {\ar"8"},
"17", {\ar"4"},
"6", {\ar"5"},
"7", {\ar"5"},
"18", {\ar"5"},
"8", {\ar"6"},
"6", {\ar"9"},
"14", {\ar"6"},
"8", {\ar"7"},
"19", {\ar"7"},
"9", {\ar"8"},
"16", {\ar"8"},
"13", {\ar"9"},
\end{xy} &

\begin{xy} 0;<.5pt,0pt>:<0pt,-.4pt>:: 
(75,50) *+{\color{darkspringgreen}{g_1^3}} ="0",
(200,250) *+{\color{red}{g_2^3}} ="1",
(325,50) *+{\color{red}{g_3^3}} ="2",
(325,175)*+{\color{red}{g_1^2}} ="3",
(200,300) *+{\color{red}{g_2^2} }="4",
(325,300) *+{\color{red}{g_3^2} }="5",
(200,150)  *+{\color{darkspringgreen}{g_1^1} }="6",
(200,375) *+{\color{red}{g_2^1}} ="7",
(75,300) *+{\color{red}{g_3^1} }="8",
(75,200) *+{\color{red}{g_0}} ="9",
(25,50) *+{\color{blue}{{g_1^3}'}} ="10",
(125,225)*+{\color{blue}{{g_2^3}'} }="11",
(325,0) *+{\color{blue}{{g_3^3}'}} ="12",
(0,125) *+{\color{blue}{{g_1^2}'}} ="13",
(275,275) *+{\color{blue}{{g_2^2}'}} ="14",
(275,100) *+{\color{blue}{{g_3^2}'}} ="15",
(0,200) *+{\color{blue}{{g_1^1}'}} ="16",
(150,325) *+{\color{blue}{{g_2^1}'}} ="17",
(25,375) *+{\color{blue}{g_3^1}'} ="18",
(0,300) *+{\color{blue}{{g_0}'}} ="19",
"2", {\ar"0"},
"0", {\ar"10"},
"0", {\ar"12"},
"0", {\ar"13"},
"0", {\ar"15"},
"0", {\ar"16"},
"0", {\ar"18"},
"0", {\ar"19"},
"1", {\ar"6"},
"8", {\ar"1"},
"11", {\ar"1"},
"14", {\ar"1"},
"15", {\ar"1"},
"3", {\ar"2"},
"12", {\ar"2"},
"5", {\ar"3"},
"3", {\ar"6"},
"15", {\ar"3"},
"5", {\ar"4"},
"4", {\ar"8"},
"17", {\ar"4"},
"6", {\ar"5"},
"7", {\ar"5"},
"18", {\ar"5"},
"6", {\ar"9"},
"6", {\ar"11"},
"6", {\ar"15"},
"8", {\ar"7"},
"19", {\ar"7"},
"9", {\ar"8"},
"16", {\ar"8"},
"13", {\ar"9"},
\end{xy}\\

\begin{xy} 0;<.5pt,0pt>:<0pt,-.4pt>:: 
(75,50) *+{\color{darkspringgreen}{g_1^3}} ="0",
(200,250) *+{\color{darkspringgreen}{g_2^3}} ="1",
(325,50) *+{\color{red}{g_3^3}} ="2",
(200,125)*+{\color{red}{g_1^2}} ="3",
(200,300) *+{\color{red}{g_2^2} }="4",
(325,300) *+{\color{red}{g_3^2} }="5",
(325,175)  *+{\color{red}{g_1^1} }="6",
(200,375) *+{\color{red}{g_2^1}} ="7",
(75,300) *+{\color{red}{g_3^1} }="8",
(75,200) *+{\color{red}{g_0}} ="9",
(25,50) *+{\color{blue}{{g_1^3}'}} ="10",
(325,100)  *+{\color{blue}{{g_2^3}'} }="11",
(325,0) *+{\color{blue}{{g_3^3}'}} ="12",
(0,125) *+{\color{blue}{{g_1^2}'}} ="13",
(200,215) *+{\color{blue}{{g_2^2}'}} ="14",
(200,0) *+{\color{blue}{{g_3^2}'}} ="15",
(0,200) *+{\color{blue}{{g_1^1}'}} ="16",
(150,325) *+{\color{blue}{{g_2^1}'}} ="17",
(25,375) *+{\color{blue}{g_3^1}'} ="18",
(0,300) *+{\color{blue}{{g_0}'}} ="19",
"2", {\ar"0"},
"0", {\ar"10"},
"0", {\ar"12"},
"0", {\ar"13"},
"0", {\ar"15"},
"0", {\ar"16"},
"0", {\ar"18"},
"0", {\ar"19"},
"1", {\ar"5"},
"6", {\ar"1"},
"8", {\ar"1"},
"1", {\ar"9"},
"14", {\ar"1"},
"3", {\ar"2"},
"12", {\ar"2"},
"6", {\ar"3"},
"3", {\ar"9"},
"3", {\ar"11"},
"5", {\ar"4"},
"4", {\ar"8"},
"17", {\ar"4"},
"5", {\ar"6"},
"7", {\ar"5"},
"18", {\ar"5"},
"9", {\ar"6"},
"11", {\ar"6"},
"15", {\ar"6"},
"8", {\ar"7"},
"19", {\ar"7"},
"9", {\ar"8"},
"16", {\ar"8"},
"13", {\ar"9"},
\end{xy} &
\begin{xy} 0;<.5pt,0pt>:<0pt,-.4pt>:: 
(75,50) *+{\color{darkspringgreen}{g_1^3}} ="0",
(200,250)*+{\color{red}{g_2^3}} ="1",
(325,50) *+{\color{red}{g_3^3}} ="2",
(198,124)  *+{\color{red}{g_1^2}} ="3",
(200,300) *+{\color{red}{g_2^2} }="4",
(325,300) *+{\color{red}{g_3^2} }="5",
(325,175) *+{\color{red}{g_1^1} }="6",
(200,375) *+{\color{red}{g_2^1}} ="7",
(75,300) *+{\color{red}{g_3^1} }="8",
(75,200) *+{\color{red}{g_0}} ="9",
(25,50) *+{\color{blue}{{g_1^3}'}} ="10",
(300,100)*+{\color{blue}{{g_2^3}'} }="11",
(325,0) *+{\color{blue}{{g_3^3}'}} ="12",
(0,125) *+{\color{blue}{{g_1^2}'}} ="13",
(200,200)  *+{\color{blue}{{g_2^2}'}} ="14",
(200,0) *+{\color{blue}{{g_3^2}'}} ="15",
(0,200) *+{\color{blue}{{g_1^1}'}} ="16",
(150,325) *+{\color{blue}{{g_2^1}'}} ="17",
(25,375) *+{\color{blue}{g_3^1}'} ="18",
(0,300) *+{\color{blue}{{g_0}'}} ="19",
"2", {\ar"0"},
"0", {\ar"10"},
"0", {\ar"12"},
"0", {\ar"13"},
"0", {\ar"15"},
"0", {\ar"16"},
"0", {\ar"18"},
"0", {\ar"19"},
"1", {\ar"5"},
"6", {\ar"1"},
"8", {\ar"1"},
"1", {\ar"9"},
"14", {\ar"1"},
"2", {\ar"3"},
"6", {\ar"2"},
"12", {\ar"2"},
"3", {\ar"6"},
"9", {\ar"3"},
"11", {\ar"3"},
"5", {\ar"4"},
"4", {\ar"8"},
"17", {\ar"4"},
"5", {\ar"6"},
"7", {\ar"5"},
"18", {\ar"5"},
"15", {\ar"6"},
"8", {\ar"7"},
"19", {\ar"7"},
"9", {\ar"8"},
"16", {\ar"8"},
"13", {\ar"9"},
\end{xy}\\
\end{array}$$
After these four mutations $g_1^3$ is the only remaining green vertex, and is the initial vertex in a 2-path through a frozen vertex for 6 vertices. However, the terminal vertex in these 2-paths form an equioriented affine subquiver with $g_1^3$ being the sink for this subquiver. The remaining mutations of $\alpha_3$ is just the mutation along the vertices of this subquiver.  Note that at each step through this part of the sequence there is a unique green vertex with a unique edge with head at mutable vertex and tail at the green vertex. Rearranging the vertices from our previous picture we obtain the following picture from which it is easy to see that the remaining mutations give a maximal green sequence.
$$
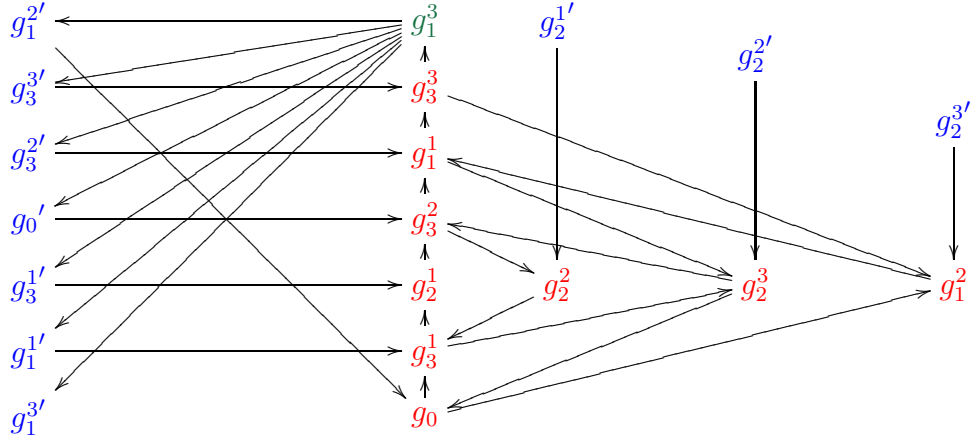
\begin{xy} 0;<1pt,0pt>:<0pt,-.5pt>:: 
(150,0) *+{\color{darkspringgreen}{g_1^3}} ="0",
(275,200) *+{\color{red}{g_2^3}} ="1",
(150,50) *+{\color{red}{g_3^3}} ="2",
(350,200) *+{\color{red}{g_1^2}} ="3",
(200,200) *+{\color{red}{g_2^2}} ="4",
(150,150) *+{\color{red}{g_3^2}} ="5",
(150,100) *+{\color{red}{g_1^1}} ="6",
(150,200) *+{\color{red}{g_2^1}} ="7",
(150,250) *+{\color{red}{g_3^1}} ="8",
(150,300) *+{\color{red}{g_0}} ="9",
(0,300) *+{\color{blue}{{g_1^3}'}} ="10",
(350,75) *+{\color{blue}{{g_2^3}'}} ="11",
(0,50) *+{\color{blue}{{g_3^3}'}} ="12",
(0,0) *+{\color{blue}{{g_1^2}'}} ="13",
(275,25) *+{\color{blue}{{g_2^2}'}} ="14",
(0,100) *+{\color{blue}{{g_3^2}'}} ="15",
(0,250) *+{\color{blue}{{g_1^1}'}} ="16",
(200,0) *+{\color{blue}{{g_2^1}'}} ="17",
(0,150) *+{\color{blue}{{g_0}'}} ="18",
(0,200) *+{\color{blue}{{g_3^1}'}} ="19",
"2", {\ar"0"},
"0", {\ar"10"},
"0", {\ar"12"},
"0", {\ar"13"},
"0", {\ar"15"},
"0", {\ar"16"},
"0", {\ar"18"},
"0", {\ar"19"},
"1", {\ar"5"},
"6", {\ar"1"},
"8", {\ar"1"},
"1", {\ar"9"},
"14", {\ar"1"},
"2", {\ar"3"},
"6", {\ar"2"},
"12", {\ar"2"},
"3", {\ar"6"},
"9", {\ar"3"},
"11", {\ar"3"},
"5", {\ar"4"},
"4", {\ar"8"},
"17", {\ar"4"},
"5", {\ar"6"},
"7", {\ar"5"},
"18", {\ar"5"},
"15", {\ar"6"},
"8", {\ar"7"},
"19", {\ar"7"},
"9", {\ar"8"},
"16", {\ar"8"},
"13", {\ar"9"},
\end{xy}$$ 
Thus $\alpha_0\alpha_1\alpha_2\alpha_3$ is a maximal green sequence for $P_3$. Our claim follows from induction on $n$. Suppose for $1 \leq k < n$ $\alpha_0\cdots\alpha_{k}$ gives a maximal green sequence for $\alpha_k$. Note that $Q_n$ has a subquiver of $P_{n-1}$ for which $\alpha_0\alpha_{n-1}$ is a green sequence. Note that the local configuration of "top nine" vertices $g_i^j$ $i=1,2,3 \, j=n-2,n-1,n$ have the exact same configuration as the "top nine" vertices of $P_3$ and the first four mutations of $\alpha_n$ exactly mimic that of $n=3$ case. (Possibly need to show this in a lemma.) Therefore we get that after the green sequence $\alpha_0\cdots\alpha_{n-1} g_3^j g_2^j g_1^{j-2} g_1^{j-1}$ we have the quiver:
$$
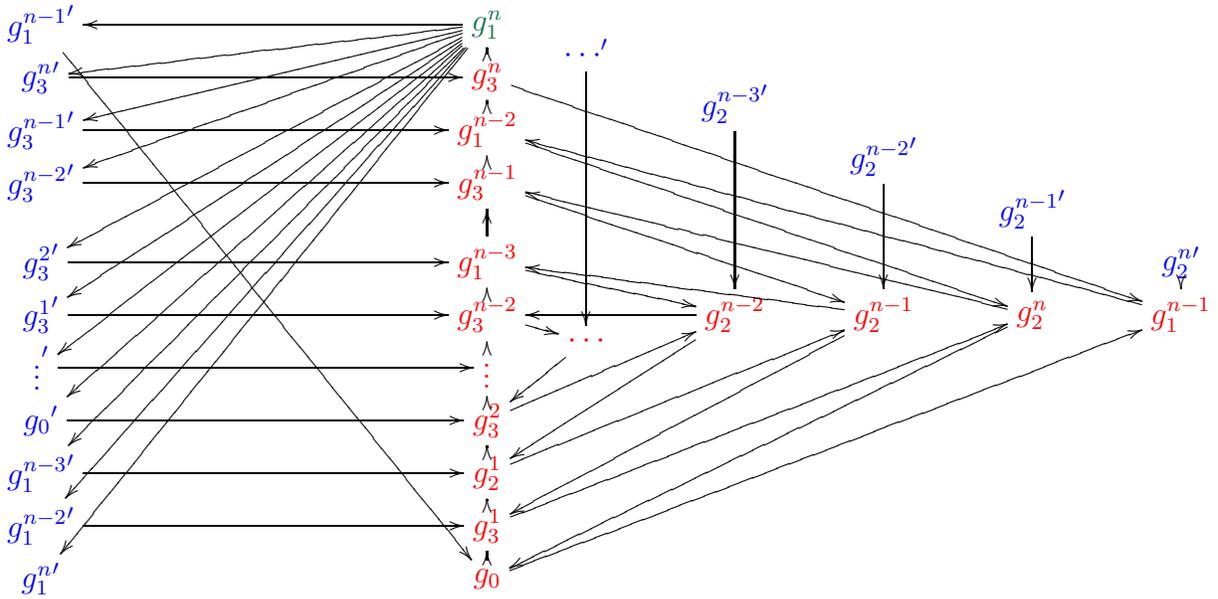
\begin{xy} 0;<.75pt,0pt>:<0pt,-.4pt>:: 
(225,0) *+{\color{darkspringgreen}{g_1^n}} ="0",
(500,275) *+{\color{red}{g_2^n}} ="1",
(225,50) *+{\color{red}{g_3^n}} ="2",
(575,275) *+{\color{red}{g_1^{n-1}}} ="3",
(425,275) *+{\color{red}{g_2^{n-1}}} ="4",
(225,150) *+{\color{red}{g_3^{n-1}}} ="5",
(225,100) *+{\color{red}{g_1^{n-2}}} ="6",
(350,275) *+{\color{red}{g_2^{n-2}}} ="7",
(225,275) *+{\color{red}{g_3^{n-2}}} ="8",
(225,225) *+{\color{red}{g_1^{n-3}}} ="9",
(275,300) *+{\color{red}{\cdots}} ="10",
(225,375) *+{\color{red}{g_3^2}} ="11",
(225,325) *+{\color{red}{\vdots}} ="12",
(225,425) *+{\color{red}{g_2^1}} ="13",
(225,475) *+{\color{red}{g_3^1}} ="14",
(225,525) *+{\color{red}{g_0}} ="15",
(0,525) *+{\color{blue}{{g_1^n}'}} ="16",
(575,225) *+{\color{blue}{{g_2^n}'}} ="17",
(0,50) *+{\color{blue}{{g_3^n}'}} ="18",
(0,0) *+{\color{blue}{{g_1^{n-1}}'}} ="19",
(500,175) *+{\color{blue}{{g_2^{n-1}}'}} ="20",
(0,100) *+{\color{blue}{{g_3^{n-1}}'}} ="21",
(0,475) *+{\color{blue}{{g_1^{n-2}}'}} ="22",
(425,125) *+{\color{blue}{{g_2^{n-2}}'}} ="23",
(0,150) *+{\color{blue}{{g_3^{n-2}}'}} ="24",
(0,425) *+{\color{blue}{{g_1^{n-3}}'}} ="25",
(350,75) *+{\color{blue}{{g_2^{n-3}}'}} ="26",
(0,225) *+{\color{blue}{{g_3^2}'}} ="27",
(0,375) *+{\color{blue}{{g_0}'}} ="28",
(275,25) *+{\color{blue}{{\cdots}'}} ="29",
(0,275) *+{\color{blue}{{g_3^1}'}} ="30",
(0,325) *+{\color{blue}{{\vdots}'}} ="31",
"2", {\ar"0"},
"0", {\ar"16"},
"0", {\ar"18"},
"0", {\ar"19"},
"0", {\ar"21"},
"0", {\ar"22"},
"0", {\ar"24"},
"0", {\ar"25"},
"0", {\ar"27"},
"0", {\ar"28"},
"0", {\ar"30"},
"0", {\ar"31"},
"1", {\ar"5"},
"6", {\ar"1"},
"14", {\ar"1"},
"1", {\ar"15"},
"20", {\ar"1"},
"2", {\ar"3"},
"6", {\ar"2"},
"18", {\ar"2"},
"3", {\ar"6"},
"15", {\ar"3"},
"17", {\ar"3"},
"5", {\ar"4"},
"4", {\ar"9"},
"13", {\ar"4"},
"4", {\ar"14"},
"23", {\ar"4"},
"5", {\ar"6"},
"9", {\ar"5"},
"24", {\ar"5"},
"21", {\ar"6"},
"7", {\ar"8"},
"9", {\ar"7"},
"11", {\ar"7"},
"7", {\ar"13"},
"26", {\ar"7"},
"8", {\ar"9"},
"8", {\ar"10"},
"12", {\ar"8"},
"30", {\ar"8"},
"27", {\ar"9"},
"10", {\ar"11"},
"29", {\ar"10"},
"11", {\ar"12"},
"13", {\ar"11"},
"28", {\ar"11"},
"31", {\ar"12"},
"14", {\ar"13"},
"25", {\ar"13"},
"15", {\ar"14"},
"22", {\ar"14"},
"19", {\ar"15"},
\end{xy} $$
Where again it is easy to check the remaining mutations will give us a maximal green sequence for $P_n$. 
\end{proof}

\begin{theorem}{\cite[Theorem 4.1]{bucher}} \label{2punc}
The quiver $Q_n^2$ has a maximal green sequence of 
\[\gamma(f_n,f_{n-1},\dots,f_1)\sigma_n\sigma_{n-1} \dots \sigma_{1}\gamma(f_2, f_1,f_3 \dots f_n) \tau_n\tau_{n-1}\dots\tau_1.\]
\end{theorem}
We now prove our main result. To assist in reading the proof we provide a running example of the maximal green sequence for $Q^5_3$.
\begin{proof}[Proof of Theorem \ref{main_v2}]
By Lemma \ref{tail} and the proof of Theorem \ref{2punc} in \cite{bucher} we know that $$\gamma(f_{n+2}f_{n+1}\cdots f_1)\sigma_n\cdots\sigma_1\alpha_0\alpha_1\cdots\alpha_{p-3}$$ is a green sequence. After performing this mutation sequence we have the that all of the vertices are red except for the vertices $f_1\cdots f_{n+2}$. Furthermore, all of these vertices except for $f_{n+1}$ form an $(n+1)$-cycle .
$$
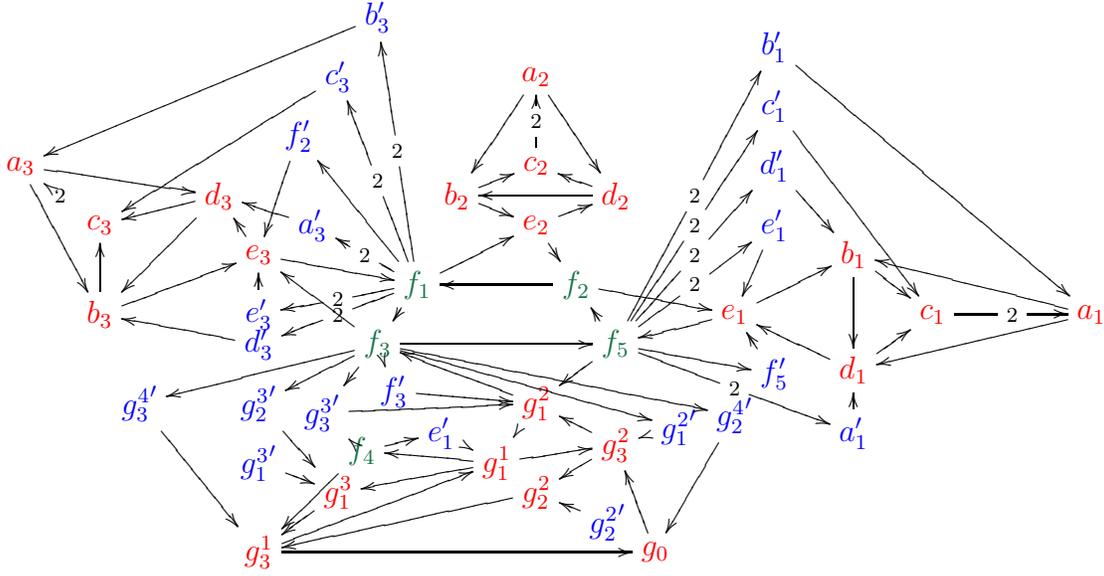
\begin{xy} 0;<.6pt,0pt>:<0pt,-.45pt>:: 
(225,275) *+{\color{darkspringgreen}{f_3}} ="0",
(350,225) *+{\color{darkspringgreen}{f_2} }="1",
(250,225) *+{\color{darkspringgreen}{f_1} }="2",
(375,275) *+{\color{darkspringgreen}{f_5}} ="3",
(215,365) *+{\color{darkspringgreen}{f_4}} ="4",
(325,325) *+{\color{red}{g_1^2} }="5",
(150,200) *+{\color{red}{e_3} }="6",
(325,175) *+{\color{red}{e_2} }="7",
(450,250) *+{\color{red}{e_1} }="8",
(375,360) *+{\color{red}{g_3^2} }="9",
(325,400) *+{\color{red}{g_2^2}} ="10",
(200,400) *+{\color{red}{g_1^3}} ="11",
(125,150) *+{\color{red}{d_3}} ="12",
(0,125) *+{\color{red}{a_3}} ="13",
(50,250) *+{\color{red}{b_3}} ="14",
(50,175) *+{\color{red}{c_3}} ="15",
(375,150) *+{\color{red}{d_2}} ="16",
(325,50) *+{\color{red}{a_2}} ="17",
(275,150) *+{\color{red}{b_2}} ="18",
(325,125) *+{\color{red}{c_2}} ="19",
(525,200) *+{\color{red}{b_1}} ="20",
(675,250) *+{\color{red}{a_1}} ="21",
(525,297) *+{\color{red}{d_1}} ="22",
(575,250) *+{\color{red}{c_1}} ="23",
(300,375) *+{\color{red}{g_1^1}} ="24",
(400,450) *+{\color{red}{g_0}} ="25",
(150,450) *+{\color{red}{g_3^1}} ="26",
(175,100) *+{\color{blue}{f_2'}} ="27",
(475,300) *+{\color{blue}{f_5'}} ="28",
(235,315) *+{\color{blue}{f_3'}} ="29",
(150,375) *+{\color{blue}{{g_1^3}'}} ="30",
(150,325) *+{\color{blue}{{g_2^3}'}} ="31",
(150,250) *+{\color{blue}{e_{3}'}} ="32",
(475,175) *+{\color{blue}{e_1'}} ="33",
(190,333) *+{\color{blue}{{g_3^3}'}} ="34",
(265,350) *+{\color{blue}{e_1'}} ="35",
(450,335) *+{\color{blue}{{g_2^4}'}} ="36",
(150,275) *+{\color{blue}{d_3'}} ="37",
(184,175) *+{\color{blue}{a_3'}} ="38",
(225,0) *+{\color{blue}{b_3'} }="39",
(200,50) *+{\color{blue}{c_3'} }="40",
(525,350) *+{\color{blue}{a_1'}} ="41",
(475,125) *+{\color{blue}{d_1'}} ="42",
(475,25) *+{\color{blue}{b_1'} }="43",
(475,75) *+{\color{blue}{c_1'} }="44",
(75,325) *+{\color{blue}{{g_3^4}'} }="45",
(370,425) *+{\color{blue}{{g_2^2}'}} ="46",
(415,345) *+{\color{blue}{{g_1^2}'} }="47",
"2", {\ar"0"},
"0", {\ar"3"},
"5", {\ar"0"},
"0", {\ar"6"},
"0", {\ar"29"},
"0", {\ar"31"},
"0", {\ar"34"},
"0", {\ar"36"},
"0", {\ar"45"},
"0", {\ar"47"},
"1", {\ar"2"},
"3", {\ar"1"},
"7", {\ar"1"},
"1", {\ar"8"},
"6", {\ar"2"},
"2", {\ar"7"},
"2", {\ar"27"},
"2", {\ar|*+{\scriptstyle 2}"32"},
"2", {\ar|*+{\scriptstyle 2}"37"},
"2", {\ar|*+{\scriptstyle 2}"38"},
"2", {\ar|*+{\scriptstyle 2}"39"},
"2", {\ar|*+{\scriptstyle 2}"40"},
"3", {\ar"5"},
"8", {\ar"3"},
"3", {\ar"28"},
"3", {\ar|*+{\scriptstyle 2}"33"},
"3", {\ar|*+{\scriptstyle 2}"41"},
"3", {\ar|*+{\scriptstyle 2}"42"},
"3", {\ar|*+{\scriptstyle 2}"43"},
"3", {\ar|*+{\scriptstyle 2}"44"},
"24", {\ar"4"},
"4", {\ar"26"},
"4", {\ar"34"},
"4", {\ar"35"},
"9", {\ar"5"},
"5", {\ar"24"},
"29", {\ar"5"},
"34", {\ar"5"},
"6", {\ar"12"},
"14", {\ar"6"},
"27", {\ar"6"},
"32", {\ar"6"},
"7", {\ar"16"},
"18", {\ar"7"},
"8", {\ar"20"},
"22", {\ar"8"},
"28", {\ar"8"},
"33", {\ar"8"},
"9", {\ar"10"},
"24", {\ar"9"},
"25", {\ar"9"},
"47", {\ar"9"},
"10", {\ar"26"},
"46", {\ar"10"},
"24", {\ar"11"},
"11", {\ar"26"},
"30", {\ar"11"},
"31", {\ar"11"},
"13", {\ar"12"},
"12", {\ar"14"},
"12", {\ar"15"},
"38", {\ar"12"},
"13", {\ar"14"},
"15", {\ar|*+{\scriptstyle 2}"13"},
"39", {\ar"13"},
"14", {\ar"15"},
"37", {\ar"14"},
"40", {\ar"15"},
"17", {\ar"16"},
"16", {\ar"18"},
"16", {\ar"19"},
"17", {\ar"18"},
"19", {\ar|*+{\scriptstyle 2}"17"},
"18", {\ar"19"},
"21", {\ar"20"},
"20", {\ar"22"},
"20", {\ar"23"},
"42", {\ar"20"},
"21", {\ar"22"},
"23", {\ar|*+{\scriptstyle 2}"21"},
"43", {\ar"21"},
"22", {\ar"23"},
"41", {\ar"22"},
"44", {\ar"23"},
"26", {\ar"24"},
"35", {\ar"24"},
"26", {\ar"25"},
"36", {\ar"25"},
"45", {\ar"26"},
\end{xy}$$
By Lemma \ref{cycle} our next section of the maximal green sequence is a green sequence for this cycle. Performing this cycle will make the $e_1,  \ldots, e_n,$ and $g_1^2$ green. 
$$
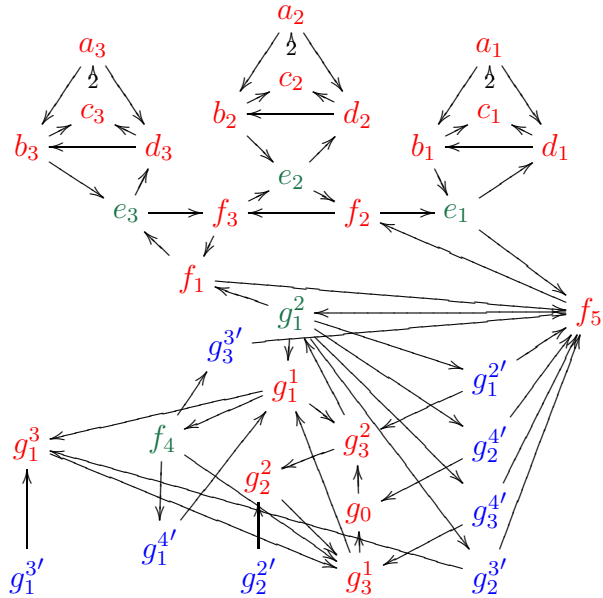
\begin{xy} 0;<.5pt,0pt>:<0pt,-.5pt>:: 
(150,150) *+{\color{red}{f_3}} ="0",
(250,150) *+{\color{red}{f_2} }="1",
(125,200) *+{\color{red}{f_1} }="2",
(425,225) *+{\color{red}{f_5}} ="3",
(102,321) *+{\color{darkspringgreen}{f_4}} ="4",
(200,226) *+{\color{darkspringgreen}{g_1^2}} ="5",
(75,150) *+{\color{darkspringgreen}{e_3}} ="6",
(200,125) *+{\color{darkspringgreen}{e_2}} ="7",
(325,150) *+{\color{darkspringgreen}{e_1}} ="8",
(250,322) *+{\color{red}{g_3^2}} ="9",
(175,350) *+{\color{red}{g_2^2}} ="10",
(0,325) *+{\color{red}{g_1^3}} ="11",
(100,100) *+{\color{red}{d_3}} ="12",
(50,25) *+{\color{red}{a_3}} ="13",
(0,100) *+{\color{red}{b_3}} ="14",
(50,75) *+{\color{red}{c_3}} ="15",
(250,75) *+{\color{red}{d_2}} ="16",
(200,0) *+{\color{red}{a_2} }="17",
(150,75) *+{\color{red}{b_2}} ="18",
(200,50) *+{\color{red}{c_2}} ="19",
(400,100) *+{\color{red}{d_1}} ="20",
(350,25) *+{\color{red}{a_1} }="21",
(300,100) *+{\color{red}{b_1}} ="22",
(350,75) *+{\color{red}{c_1}} ="23",
(196,281) *+{\color{red}{g_1^1}} ="24",
(251,377) *+{\color{red}{g_0}} ="25",
(252,426) *+{\color{red}{g_3^1}} ="26",
(0,425) *+{\color{blue}{{g_1^3}'}} ="27",
(350,425) *+{\color{blue}{{g_2^3}'}} ="28",
(150,250) *+{\color{blue}{{g_3^3}'}} ="29",
(100,403) *+{\color{blue}{{g_1^4}'}} ="30",
(350,325) *+{\color{blue}{{g_2^4}'}} ="31",
(350,375) *+{\color{blue}{{g_3^4}'}} ="32",
(175,425) *+{\color{blue}{{g_2^2}'}} ="33",
(350,275) *+{\color{blue}{{g_1^2}'}} ="34",
"1", {\ar"0"},
"0", {\ar"2"},
"6", {\ar"0"},
"0", {\ar"7"},
"3", {\ar"1"},
"7", {\ar"1"},
"1", {\ar"8"},
"2", {\ar"3"},
"5", {\ar"2"},
"2", {\ar"6"},
"3", {\ar"5"},
"8", {\ar"3"},
"28", {\ar"3"},
"29", {\ar"3"},
"31", {\ar"3"},
"32", {\ar"3"},
"34", {\ar"3"},
"24", {\ar"4"},
"4", {\ar"26"},
"4", {\ar"29"},
"4", {\ar"30"},
"9", {\ar"5"},
"5", {\ar"24"},
"5", {\ar"28"},
"5", {\ar"31"},
"5", {\ar"32"},
"5", {\ar"34"},
"6", {\ar"12"},
"14", {\ar"6"},
"7", {\ar"16"},
"18", {\ar"7"},
"8", {\ar"20"},
"22", {\ar"8"},
"9", {\ar"10"},
"24", {\ar"9"},
"25", {\ar"9"},
"34", {\ar"9"},
"10", {\ar"26"},
"33", {\ar"10"},
"24", {\ar"11"},
"11", {\ar"26"},
"27", {\ar"11"},
"28", {\ar"11"},
"13", {\ar"12"},
"12", {\ar"14"},
"12", {\ar"15"},
"13", {\ar"14"},
"15", {\ar|*+{\scriptstyle 2}"13"},
"14", {\ar"15"},
"17", {\ar"16"},
"16", {\ar"18"},
"16", {\ar"19"},
"17", {\ar"18"},
"19", {\ar|*+{\scriptstyle 2}"17"},
"18", {\ar"19"},
"21", {\ar"20"},
"20", {\ar"22"},
"20", {\ar"23"},
"21", {\ar"22"},
"23", {\ar|*+{\scriptstyle 2}"21"},
"22", {\ar"23"},
"26", {\ar"24"},
"30", {\ar"24"},
"26", {\ar"25"},
"31", {\ar"25"},
"32", {\ar"26"},
\end{xy}$$
%\caption{ }
%\label{step2}
%\end{figure}

After mutating at $f_{n+1}$ and $g_1^{p-4}$ we have a similar  situation as we did at the end of the proof of Lemma \ref{tail}. Note that in our running example of $Q^5_3$ that $g_1^{p-4}=g_1^1,$ and is not shown to be green in any picture. The remaining vertices that we mutate along in the $\beta$ subsequence form an equioriented affine subquiver. It is easy to follow that this is a green sequence that will make the $P_{p-3}$ subquiver of $Q_n^p$ red.
$$
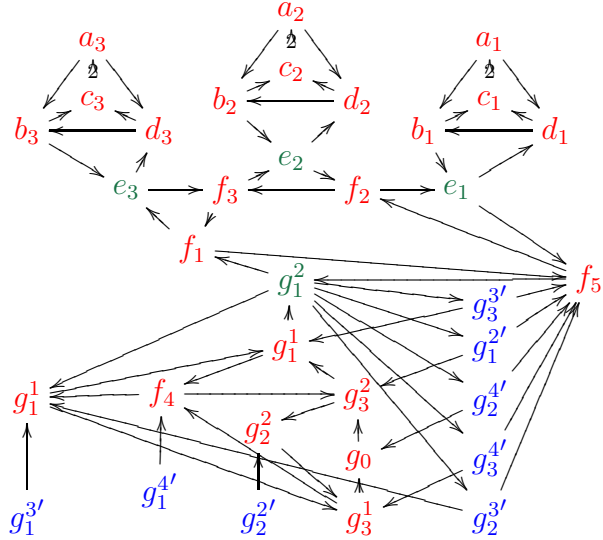
\begin{xy} 0;<.5pt,0pt>:<0pt,-.45pt>:: 
(150,150) *+{\color{red}{f_3}} ="0",
(250,150) *+{\color{red}{f_2} }="1",
(125,200) *+{\color{red}{f_1} }="2",
(425,225) *+{\color{red}{f_5}} ="3",
(102,321) *+{\color{red}{f_4}} ="4",
(200,226) *+{\color{darkspringgreen}{g_1^2}} ="5",
(75,150) *+{\color{darkspringgreen}{e_3}} ="6",
(200,125) *+{\color{darkspringgreen}{e_2}} ="7",
(325,150) *+{\color{darkspringgreen}{e_1}} ="8",
(250,322) *+{\color{red}{g_3^2}} ="9",
(175,350) *+{\color{red}{g_2^2}} ="10",
(0,325) *+{\color{red}{g_1^1}} ="11",
(100,100) *+{\color{red}{d_3}} ="12",
(50,25) *+{\color{red}{a_3}} ="13",
(0,100) *+{\color{red}{b_3}} ="14",
(50,75) *+{\color{red}{c_3}} ="15",
(250,75) *+{\color{red}{d_2}} ="16",
(200,0) *+{\color{red}{a_2} }="17",
(150,75) *+{\color{red}{b_2}} ="18",
(200,50) *+{\color{red}{c_2}} ="19",
(400,100) *+{\color{red}{d_1}} ="20",
(350,25) *+{\color{red}{a_1} }="21",
(300,100) *+{\color{red}{b_1}} ="22",
(350,75) *+{\color{red}{c_1}} ="23",
(196,281) *+{\color{red}{g_1^1}} ="24",
(251,377) *+{\color{red}{g_0}} ="25",
(252,426) *+{\color{red}{g_3^1}} ="26",
(0,425) *+{\color{blue}{{g_1^3}'}} ="27",
(350,425) *+{\color{blue}{{g_2^3}'}} ="28",
(350,245) *+{\color{blue}{{g_3^3}'}} ="29",
(100,403) *+{\color{blue}{{g_1^4}'}} ="30",
(350,325) *+{\color{blue}{{g_2^4}'}} ="31",
(350,375) *+{\color{blue}{{g_3^4}'}} ="32",
(175,425) *+{\color{blue}{{g_2^2}'}} ="33",
(350,280) *+{\color{blue}{{g_1^2}'}} ="34",
"1", {\ar"0"},
"0", {\ar"2"},
"6", {\ar"0"},
"0", {\ar"7"},
"3", {\ar"1"},
"7", {\ar"1"},
"1", {\ar"8"},
"2", {\ar"3"},
"5", {\ar"2"},
"2", {\ar"6"},
"3", {\ar"5"},
"8", {\ar"3"},
"28", {\ar"3"},
"29", {\ar"3"},
"31", {\ar"3"},
"32", {\ar"3"},
"34", {\ar"3"},
"4", {\ar"9"},
"4", {\ar"11"},
"24", {\ar"4"},
"26", {\ar"4"},
"30", {\ar"4"},
"5", {\ar"11"},
"24", {\ar"5"},
"5", {\ar"28"},
"5", {\ar"29"},
"5", {\ar"31"},
"5", {\ar"32"},
"5", {\ar"34"},
"6", {\ar"12"},
"14", {\ar"6"},
"7", {\ar"16"},
"18", {\ar"7"},
"8", {\ar"20"},
"22", {\ar"8"},
"9", {\ar"10"},
"9", {\ar"24"},
"25", {\ar"9"},
"34", {\ar"9"},
"10", {\ar"26"},
"33", {\ar"10"},
"11", {\ar"24"},
"11", {\ar"26"},
"27", {\ar"11"},
"28", {\ar"11"},
"13", {\ar"12"},
"12", {\ar"14"},
"12", {\ar"15"},
"13", {\ar"14"},
"15", {\ar|*+{\scriptstyle 2}"13"},
"14", {\ar"15"},
"17", {\ar"16"},
"16", {\ar"18"},
"16", {\ar"19"},
"17", {\ar"18"},
"19", {\ar|*+{\scriptstyle 2}"17"},
"18", {\ar"19"},
"21", {\ar"20"},
"20", {\ar"22"},
"20", {\ar"23"},
"21", {\ar"22"},
"23", {\ar|*+{\scriptstyle 2}"21"},
"22", {\ar"23"},
"29", {\ar"24"},
"26", {\ar"25"},
"31", {\ar"25"},
"32", {\ar"26"},
\end{xy}
$$
%\caption{}
%\label{step3}
%\end{figure}

 Finally, the only remaining green vertices are $e_i$ for $i=1,\ldots,n$. By inspection of the local configuration of the frozen vertices we see that this is the exact same configuration as in the twice punctured case. Therefore it follows from the proof of Theorem \ref{2punc}. That $\tau_i$ is a green sequence for our quiver, and concluding the proof that our green sequence is maximal. 
\end{proof}

\section{Future Interests}

In light of Theorem \ref{main2} we know of particular triangulations of surfaces that have maximal green sequences. Muller in \cite{muller} gives an example of a quiver with a maximal green sequence and a mutation equivalent quiver which does not have a maximal green sequence. The question of whether every triangulation of a surface has a maximal green sequence still remains open in many cases. It is worthwhile to point out that the example given in \cite{muller} is not one of a quiver that is associated to a surface. It is the belief of the authors that when dealing with cluster algebras that arise from surfaces that the existence of a maximal green sequence may in fact be a mutation invariant. 
\begin{conj}
 Let $Q$ be a quiver associated to a surface. If $Q$ exhibits a maximal green sequence, then any $Q'$ mutation equivalent to $Q$ also exhibits a maximal green sequence. 
\end{conj}
 
The question of existence of maximal green sequences for quivers that are not associated to surfaces is still largely open.

%\section*{Acknowledgements}
%This paper would not be possible without the helpful insight from M. Yakimov, and the LSU cluster algebra group.

%\begin{acknowledgements}
%\end{acknowledgements}



\end{document}